\documentclass{ifacconf}

\usepackage{graphicx}      
\usepackage{natbib}        
\usepackage{xcolor}
\usepackage{mathtools,amsmath,amssymb,amsfonts}
\usepackage{adjustbox}
\usepackage{enumitem}

\newcommand{\myDots}{\hbox to 0.9em{.\hss.\hss.}}

\newcommand{\R}{\mathbb{R}}{}
\newcommand{\N}{\mathbb{N}}

\newcommand{\cC}{\mathcal{C}}

\newcommand{\cH}{\mathcal{H}}

\newcommand{\M}{\Sigma}

\newcommand{\wi}{\overline {\imath}}

\newcommand{\cF}{\mathcal{F}}

\newcommand{\wt}{\widetilde}

\DeclareMathOperator*{\argmin}{argmin}

\makeatletter
\newsavebox\myboxA
\newsavebox\myboxB
\newlength\mylenA

\newcommand*\pbar[1]{%
  \hbox{%
     \vbox{%
      \hrule height 0.7pt 
      \kern0.35ex
      \hbox{%
         \kern-0.0em
         \ensuremath{#1}%
         \kern-0.0em
      }%
     }%
  }%
} 

\newcommand*\xbar[2][0.75]{%
    \sbox{\myboxA}{$\m@th#2$}%
    \setbox\myboxB\null
    \ht\myboxB=\ht\myboxA%
    \dp\myboxB=\dp\myboxA%
    \wd\myboxB=#1\wd\myboxA
    \sbox\myboxB{$\m@th\pbar{\copy\myboxB}$}
    \setlength\mylenA{\the\wd\myboxA}
    \addtolength\mylenA{-\the\wd\myboxB}%
    \ifdim\wd\myboxB<\wd\myboxA%
       \rlap{\hskip 0.5\mylenA\usebox\myboxB}{\usebox\myboxA}%
    \else
        \hskip -0.5\mylenA\rlap{\usebox\myboxA}{\hskip 0.5\mylenA\usebox\myboxB}%
    \fi}
\makeatother

\allowdisplaybreaks

\begin{document}
\begin{frontmatter}

\title{Feedback Stabilization of Switched Systems: Memory is not needed.\thanksref{footnoteinfo}} 

\thanks[footnoteinfo]{T. Alves Lima was supported by CNPq (Conselho Nacional de Desenvolvimento Científico e Tecnológico) via the grant number 443674/2024-8. This paper was partly written while T. Alves Lima was with the Department of Electrical Engineering, Federal University of Ceará
(UFC), Fortaleza, 60020-181, CE, Brazil.}

\author[First]{Thiago Alves Lima} 
\author[Second]{Matteo Della Rossa} 
\author[Third]{Antoine Girard}

\address[First]{Systems Engineering Division, Aeronautics Institute of Technology (ITA), Fortaleza, 60415-513, CE, Brazil (e-mail: thiago.lima@gp.ita.br).}
\address[Second]{Department of Electronics and Telecommunications, Politecnico di Torino, corso Duca degli Abruzzi, 24, Torino, Italy (e-mail: matteo.dellarossa@polito.it).}
\address[Third]{Université Paris-Saclay, CNRS, CentraleSupélec, Laboratoire des Signaux et Systèmes, 91190, Gif-sur-Yvette, France (e-mail: antoine.girard@centralesupelec.fr).}

\begin{abstract}   
A long-standing assumption in the literature on switched linear systems is that static, homogeneous of degree one feedbacks form the most general class of controllers necessary and sufficient for stabilization. In this paper, we provide a rigorous justification. More specifically, we prove by construction that if a switched linear system admits a stabilizing full-information controller, with access to the entire history of states and switching signals, then a memoryless and homogeneous of degree one stabilizing controller also exists. Specifically, in the mode-independent setting the controller can be chosen to depend only on the current state, and in the mode-dependent setting only on the current state and active mode. Our results thus show that dynamic controllers offer no additional stabilizing capability for switched linear systems, formally validating this folklore claim.
\end{abstract}


\end{frontmatter}

\section{Introduction}

Switched systems are dynamical systems whose evolution depends on a family of modes (or subsystems) and on a switching signal that determines which mode is active at each instant. They arise in numerous engineering applications and exhibit intrinsically rich behaviors: for example, switching among a collection of individually stable modes can render the overall system unstable, and, conversely, appropriate switching may stabilize a family of unstable modes. For further details and classical results, see, for instance,~\citep{Lib03}.

A fundamental distinction in the study of switched systems concerns whether the switching signal is treated as external and arbitrary, or as a control input to be designed in order to achieve a desired objective. In the latter case, many strategies have been proposed to achieve stabilization by appropriate design of the switching law; see, for example,~\citep{FiaGir16,GerCol06,WICKS1998140}. Such notion of stabilization via switched signal has been also studied from a complementary perspective in~\citep{JunMas17,DettJun20}, introducing and analyzing the properties of the so called lower- and stabilizability spectral radius.

In a different line of work, several authors have also considered the joint synthesis of the switching signal and a continuous control input. For contributions along this direction, see, for instance,~\citep{HuMaLIn08,LinAnt08,ZhangAbate09,FIACCHINI2017181}.

In this paper, our focus is on the setting in which the switching signal is external and arbitrary, and where one seeks to design an additional control input to ensure suitable closed-loop behavior, uniformly on the family of arbitrary switching. This scenario has been explored, for instance, in~\cite{BlaMia03,BlaMiaSav07,LeeKha09,Lee06,DelRosAlv24}, and poses significant challenges owing to the need to provide certificates that work for \textit{all possible switching signals}.

In~\citep{BlaMia03,BlaMiaSav07}, polyhedral Lyapunov functions are employed to derive bilinear matrix (in-)equalities that enable the computation of either piecewise-linear static controllers or linear dynamic controllers. These controllers depend on the value of the system state and may either adapt to the current mode (the so-called ``mode-dependent" case) or rely solely on the current state (which we refer to as ``mode-independent" controllers).\footnote{In \cite{BlaMia03,BlaMiaSav07} and also in \cite{DelRosAlv24}, these controllers are called “robust”. Here we use the term “mode-independent’’ instead, reserving “robust’’ for when considering uncertainties in the subsystem models.}
In~\citep{LeeKha09,Lee06}, multiple quadratic Lyapunov functions are used to derive linear stabilizing controllers that depend not only on the current state but also on a finite memory of past modes. More recently,~\cite{DelRosAlv24} employed graph-based Lyapunov functions to derive linear matrix inequalities for the synthesis of stabilizing piecewise-linear control laws, thereby connecting the framework of path-complete Lyapunov functions (see~\cite{AhmJun:14}) with earlier min–max quadratic Lyapunov approaches developed for stability analysis. 

Those different feedback strategies lead to several questions regarding which classes of feedback maps should be considered, especially in terms of their stabilizability power. In particular, controllers with ``memory'' (of past states and/or switching modes) naturally include dynamical controllers of the type used in~\cite{BlaMia03,BlaMiaSav07}, which motivates the question of whether allowing memory actually enlarges the set of stabilizing controllers. More specifically, we aim to address the following question:

\textit{If a switched linear system is stabilizable by a controller that has access to the \emph{entire} past of states and switching modes, does this imply the existence of a stabilizing controller that \emph{does not} use such memory?}

In this paper, we show that the answer is affirmative. We provide a rigorous justification, for both the mode-dependent and mode-independent cases, that infinite-memory stabilizability implies stabilizability by a memoryless feedback law. In addition, we establish that the resulting memoryless stabilizing controllers can always be chosen to be homogeneous of degree one. Such homogeneity allows us also to prove that the asymptotic stability of the resulting closed-loop is exponential and robust (from exogenous disturbances), in the sense introduced in~\cite{KellTeel}. The results in this paper complement those in~\cite{lima2025}, where a hierarchy of \emph{graph-based} sufficient linear matrix inequality (LMI) conditions is shown to be asymptotically necessary for the design of piecewise linear feedback controllers based on piecewise quadratic Lyapunov functions.

\subsection{Organization and notation}

The rest of this paper is organized as follows. In Section~\ref{sec:preliiminaries} we introduce in detail the class of systems under consideration and review the relevant stability and stabilizability notions. The main results are presented in Section~\ref{sec:main}. Finally, Section~\ref{sec:conc} provides concluding remarks and outlines directions for future work.

\textbf{Notation:} We denote by $\N$ the set of natural numbers including $\{0\}$, by $\N_+$ the set of natural numbers excluding $\{0\}$. Given a countable set of symbols $\Sigma$, referred to as the \emph{alphabet}. Given $T\in \N$, we denote by $\Sigma^T$ the set of strings of length $T$ in $\Sigma$, and we denote its elements by $\wi=(i_0,\dots, i_{T-1})\in \Sigma^T$.
 We use the convention that $\Sigma^0=\{ \bf \epsilon\}$ where ``$\bf \epsilon$'' is an auxiliary symbol, which stands for ``empty string''. Given $\wi=(i_0,\dots, i_{T-1})\in \Sigma^T$, we denote by $|\wi|=T$ its \emph{length}.
 With $\displaystyle \Sigma^\star:=\bigcup_{T\in \N}\Sigma^T$ we denote the \emph{Kleene closure} of $\Sigma$, that is, the set of all finite-length strings in $\Sigma$.
 \\
 With $\Sigma^\omega$ we denote the \emph{$\omega$-closure} of $\Sigma$, i.e. the set of all the (infinite) sequences in $\Sigma$; more precisely, $\Sigma^\omega:=\{\sigma:\N\to \Sigma\}$.

Given a vector $x\in\R^n$, with $\|x\|$ we denote the Euclidean norm (a.k.a. $2$ norm).
 A function $U:\R^n\to \R$ is said to be \emph{positive definite} if $U(0)=0$ and $U(x)>0$ for all $x\neq 0$. It is said to be radially unbounded if $\lim_{\lambda \to +\infty} U(\lambda x)=+\infty$ for all $x\neq 0$.
Given $p\in \N$, a function $\Psi:\R^n\to \R^m$ is said to be \emph{homogeneous of degree $p$} if
\[
\Psi(\lambda x)=\lambda^p \Psi(x),\;\;\;\forall \lambda\in \R,\;\;\forall x\in \R^n.
\]
It is said to be \emph{absolutely homogeneous of degree 
p} if 
\[
\Psi(\lambda x)=|\lambda|^p \Psi(x),\;\;\;\forall \lambda\in \R,\;\;\forall x\in \R^n.
\]

\section{Preliminaries}\label{sec:preliiminaries}

In this section we provide the formal definition of the  class of systems studied in this paper, together with the considered stabilization notions.

\subsection{Definitions and robust stabilizability notions}
Let us fix a countable alphabet $\Sigma$, and a set  $\cF=\{(A_i,B_i)\in \R^{n\times n}\times \R^{n\times m}\;\vert\;i\in \M\}$, we study the \emph{discrete-time control switched system} defined by
\begin{equation}\label{eq:SwitchedSystemInput}
x(k+1)=A_{\sigma(k)}x(k)+B_{\sigma(k)}u(k),\;\;\;k\in \N,
\end{equation}
where $\sigma:\N\to \M$ (equivalently, $\sigma\in \M^\omega$) is the \emph{switched signal} and $u:\N\to \R^m$ is a control input.
The switching signal $\sigma:\N\to\M$ is regarded as arbitrary and cannot be modified by the user, who must achieve stabilization via the control input $u:\N\to\R^m$. In particular, we are interested in \emph{feedback} stabilization: we would like to analyze the existence (or not) of stabilizing feedback maps that depend, broadly speaking, on the ``behavior of the system''.
In particular, one can imagine several frameworks, notably:
\begin{itemize}[leftmargin=*]
\item Feedback maps that depend only on the current \emph{state},
\item Feedback maps that depend only on the current \emph{state} and current \emph{switching mode},
\item Feedback maps that possibly depend also on the \emph{past} behavior (a.k.a. the memory) of \emph{state and/or switching modes}.
\end{itemize}

To formalize the aforementioned notions of feedback stabilization for~\eqref{eq:SwitchedSystemInput} we need to introduce a general notion of ``dynamic systems depending on the past''. For a general framework of such class of systems we refer to~\cite{DelRosJung24}.
\begin{defn}\label{defn:MemoryDepSystems}
Given a continuous state-space $\R^n$ and an alphabet $\Sigma$, let us consider the set
\begin{equation*}
    \mathcal{H} := \Big\{
(\overline x,\wi)\in (\mathbb R^n)^\star \times \Sigma^\star
\;\big|\;
|\overline x| = |\wi|
\Big\}
\end{equation*}
and a function
\[
f: \mathcal{H} \to \mathbb R^n.
\]
The \emph{(memory-dependent) dynamical system induced by $f$} is defined by the recursion equation
\begin{equation} \label{eq:Recursive equation}
\begin{cases}
x(k+1)=f(x(k),\,\dots\, x(0);\,\sigma(k),\,\dots\,, \sigma(0)),\\
x(0)=x_0,\\
\sigma\in \Sigma^\omega.
\end{cases}
\end{equation}
It can be seen that, given any $x_0\in \R^n$ and any $\sigma\in \Sigma^\omega$, there exists a unique solution to~\eqref{eq:Recursive equation}, that we denote to
$\phi_f(\cdot, x_0,\sigma):\N\to \R^n$.
\end{defn}
We now introduce the considered notion of stability
for the class of systems introduced in Definition~\ref{defn:MemoryDepSystems}.
\begin{defn}\label{defn:generalStabilityDefn}
Let us consider $f:\mathcal{H} \to \mathbb R^m$. We say the system~\eqref{eq:Recursive equation} is \emph{uniformly exponentially stable} (UES) if
there exist $M > 0$ and $\gamma\in [0,1)$ such that
\[
\|\phi_f(k,\sigma,x_0)\|\leq M\gamma^k\|x_0\|,\;\;\;\forall x_0\in \R^n,\;\forall \sigma\in \Sigma^\omega,\;\forall \;k\in \N.
\]
\end{defn}
Note that ``uniformity'' in such definition involves both the initial state $x_0\in \R^n$ and the switching signal $\sigma\in \Sigma^\omega$ that can be both arbitrary.

This general definition can now be used to introduce the \emph{feedback stabilizability} notions for~\eqref{eq:SwitchedSystemInput} considered in this manuscript.
First, consider the set
\[
\mathcal{H}_{-}
:= 
\Big\{
(\overline x,\wi)\in (\mathbb R^n)^\star \times \Sigma^\star
\;\big|\;
|\overline x| = |\wi| + 1
\Big\}.
\]
In contrast to the set $\mathcal{H}$, which contains equal-length pairs, the set $\mathcal{H}_{-}$ captures the natural information pattern in which the controller knows the \textit{current and past} states but \textit{only the past switching} modes. This distinction is needed for the next definition.

\begin{defn}\label{defn:StabNot}
System~\eqref{eq:SwitchedSystemInput} is said to be
\begin{enumerate}[leftmargin=0.7cm]
\item[(1)]\emph{Mode-independent feedback stabilizable} (IFS), 
if there exists  $\Phi:\R^n \to \R^m$ such that, considering $f^{\Phi}:\mathcal{H} \to  \R^n$ defined by
\[
f^\Phi(x_k,\dots, x_0;\;i_k,\dots, i_0):=A_{i_k}x_k+B_{i_k}\Phi(x_k)\,
\] 
for all $(x_k,\dots,x_0;\;i_k,\dots, i_0)\in (\mathbb R^n)^\star \times \Sigma^\star$,
the dynamical system induced by $f^\Phi$ is UES.
\item[(1)$^d$] \emph{Mode-dependent feedback stabilizable} (DFS), 
if there exists  $\Phi_d:\Sigma\times \R^n \to \R^m$ such that, considering $f^{\Phi_d}: \mathcal{H} \to \R^n$ defined by
\[
f^{\Phi_d}(x_k,\dots x_0;\;i_k,\dots, i_0):=A_{i_k}x_k+B_{i_k}\Phi_d(i_k,x_k)\,
\] 
for all $(x_k,\dots,x_0;\;i_k,\dots, i_0)\in (\mathbb R^n)^\star \times \Sigma^\star$,
the dynamical system induced by $f^{\Phi_d}$ is UES.
\end{enumerate}
\begin{enumerate}[leftmargin=0.7cm]
\item[(2)] \emph{Current-mode-independent memory-feedback stabilizable} (IFS$_{m}$), 
if there exists $\Psi :
\mathcal{H}_{-} \to \mathbb R^m$ such that, considering $f^{\Psi}: \mathcal{H} \to \R^n$ defined by
\[
\begin{aligned}
f^\Psi&(i_k,\dots, i_0,x_k,\dots x_0):=\\&A_{i_k}x_k+B_{i_k}\Psi(x_k,\dots,x_0;\;i_{k-1},\dots, i_0)\,
\end{aligned}
\] 
for all $(x_k,\dots,x_0;\;i_k,\dots, i_0)\in (\mathbb R^n)^\star \times \Sigma^\star$, the dynamical system induced by $f^\Psi$ is UES.
\item[(2)$^d$] \emph{Current-mode-dependent memory-feedback stabilizable} (DFS$_{m}$), 
if there exists  $\Psi_d:\mathcal{H} \to \mathbb R^m$ such that, considering $f^{\Psi_d}:\mathcal{H} \to \mathbb R^m$ defined by
\[
\begin{aligned}
f^{\Psi_d}&(i_k,\dots, i_0,x_k,\dots x_0):=\\&A_{i_k}x_k+B_{i_k}\Psi_d(x_k,\dots,x_0;\;i_{k},\dots, i_0)\,
\end{aligned}
\] 
for all $(x_k,\dots,x_0;\;i_k,\dots, i_0)\in (\mathbb R^n)^\star \times \Sigma^\star$, the dynamical system induced by $f^{\Psi_d}$ is UES.

\end{enumerate}
\end{defn}
\begin{rem}
Let us discuss the introduced notions of stabilizability. In Item~(1) it is required that the system can be feedback-stabilized knowing only the value of the current state. The map $\Phi:\R^n \to \R^m$ is thus called a \emph{static mode-independent state-feedback}. In Item~(1)$^d$ we ask for the existence of a feedback map depending on both the current state and the switching value. Such a map $\Phi_d:\Sigma \times \R^n \to \R^m$ is thus called a \emph{static mode-dependent state-feedback.} It is well known that the notion of stabilizability in~(1) is strictly more demanding than the one in~(1)$^d$. For an example of a system that is DFS but not IFS we refer to~\citep[Example 5.1]{BlaMiaSav07}.

In the definition introduced in Item~(2) the feedback map may in fact depend also on the \emph{history/past} values of the state–signal pair, as well as on the current state. It cannot depend on the current switching mode. For this reason such a map $\Psi:\cH_-\to \R^m$ is called a \emph{current-mode-independent controller with state-signal memory}. In Item~(2)$^d$ the feedback controller $\Psi^d:\cH\to \R^m$ may also depend on the current mode, and is thus called a \emph{current-mode-dependent controller with state-signal memory}. As before, it is readily seen that the notion in (2) is strictly more demanding than the one in (2)$^d$. Let us note that this formalism (in which the controller can depend also on past values of the state-signal pairs) allows us, in a succinct manner, to model the case of \emph{dynamic feedback controllers}. Indeed, in such a case the dependence of the controller on past values of the state–signal pair is defined by an additional dynamical system (the dynamic controller itself). We do not develop this idea further, as it is not the main message of this manuscript.
\end{rem}

\section{Main result}\label{sec:main}
Looking at the stabilizability notions introduced 
in the previous section, one can postulate the following claim: allowing the controller to “use’’ the past values of the state–signal pairs might provide additional flexibility to the control design, thus leading to a less restrictive notion of stabilizability. In this section we show that this is \emph{not} the case. It turns out that, at least from a theoretical point of view, the stabilizability notion in~Item~(1) in Definition~\ref{defn:StabNot} is \emph{equivalent} to the notion in Item~(2) (and similarly for (1)$^d$ and~(2)$^d$).

\begin{thm} \label{thm:main}
Let us consider a finite alphabet $\Sigma$ and $\cF=\{(A_i,B_i)\in \R^{n\times n}\times \R^{n\times m}\;\vert\;i\in \M\}$. Then
\begin{enumerate}[leftmargin=*]
\item[(A)] System~\eqref{eq:SwitchedSystemInput} is IFS if and only if it is  IFS$_{m}$,\\
\item[(B)] System~\eqref{eq:SwitchedSystemInput} is DFS if and only if it is  DFS$_{m}$.
\end{enumerate}
\end{thm}
\begin{pf}
The ``\emph{only if}'' part is a straightforward consequence of the fact that, in both cases, a  static feedback controller can be considered to be a controller with state-signal memory (but that actually does not explicitly use the past of state-signal pairs). In the following, we focus on the ``\emph{if}'' part. We detail the mode-independent case (i.e. Item~\emph{(A)}) and show how the mode-dependent case can be proved along similar lines.

We denote the set of current-mode-independent controllers with state-signal memory by \[
\mathcal C_m:=\left \{\chi:\cH_-\to \R^m\right \}.
\]
Let us assume that the system is IFS$_{m}$. By definition,  there exist $C > 1$, $\gamma\in [0,1)$ and a controller $\Psi\in \mathcal C_m$  such that 
\begin{equation}
\label{eq:exp_memory}
|\phi(k,\sigma,x_0,\Psi)| \le C \gamma^k |x_0|,
\end{equation}
for all $x_0\in \R^n$, for all $ 
\sigma \in \Sigma^\omega$, and for all $k\in \N$, where $\phi(k,\sigma,x_0,\Psi)$ denotes the solution corresponding to the closed-loop (memory-dependent) system.

Let us consider the function $V:\R^n \to \R$ defined for all $x_0 \in \R^n$ by
\begin{equation}
\label{eq:lyap}
V(x_0) = \inf_{\Psi\in \mathcal C_m} \sup_{\sigma \in  \Sigma^\omega} \sum_{k=0}^{\infty} 
\|\phi(k,\sigma,x_0,\Psi)\|.    
\end{equation}
From \eqref{eq:exp_memory}, it follows that
\begin{equation}
\label{eq:lyap1}
\|x_0\| \le V(x_0) \le \frac{C}{1-\gamma} \|x_0\|,\; \forall x_0 \in \R^n.
\end{equation}
Hence, $V$ is positive definite and radially unbounded.

We now prove that $V$ is absolutely homogeneous of degree~$1$. Consider $x_0\in \R^n$, $\lambda \ne 0$ and let $z_0=\lambda x_0$. 
Consider any $\Psi \in \mathcal C_m$ and let $\Psi'\in \mathcal C_m$ be defined by
$$
\begin{aligned}
\Psi'&(y_k,\dots,y_0;\,i_{k-1},\dots,i_0):=\\&
\lambda
\Psi\left(\frac{y_k}{\lambda},\dots,\frac{y_0}{\lambda};\,i_{k-1},\dots,i_0\right)
\end{aligned}
$$
for all $(i_{k-1},\dots, i_0)\in \M^k$, all $(y_k,\dots, y_0)\in \R^{n\times (k+1)}$, and all $k\in \N$.
One can easily check that, for all $\sigma\in \M^\omega$ and for all $k\in \N$, we have
$\phi(k,\sigma,z_0,\Psi')=\lambda \phi(k,\sigma,x_0,\Psi)$.
Indeed, fix $\sigma \in\M^\omega$,  for $k=0$ we trivially have
$\phi(0,\sigma,z_0,\Psi')=z_0=\lambda x_0=\lambda \phi(0,\sigma,x_0,\Psi)$.
For $k=1$ we have: 
\[
\begin{aligned}
\phi&(1,\sigma,z_0,\Psi')= A_{\sigma(0)}z_0+B_{\sigma(0)}\Psi'(z_0,\epsilon)\\&=A_{\sigma(0)}\lambda x_0+B_{\sigma(0)}\Psi'(\lambda x_0,\epsilon)\\&=A_{\sigma(0)}\lambda x_0+B_{\sigma(0)}\lambda\Psi(x_0,\epsilon)\\&=\lambda\left (A_{\sigma(0)} x_0+B_{\sigma(0)}\Psi(x_0,\epsilon)\right)=
\lambda \phi(1,\sigma,x_0,\Psi)
\end{aligned}
\]
and then we can proceed by induction on $k\in \N$.

It follows from the definition of $V$ in~\eqref{eq:lyap}
that $V(z_0) \le |\lambda| V(x_0)$. Reversing the role of $x_0$ and $z_0$ we can show similarly that
$V(x_0) \le V(z_0)/|\lambda|$.\\
Hence, $V(\lambda x_0)=|\lambda| V(x_0)$ for all $x_0 \in \R^n$, $\lambda \ne 0$, which extends to all $\lambda \in \R$ since clearly $V(0)=0$.

We now prove that $V$ is sub-additive. Consider $z_1, z_2 \in \R^n$ and let $x_0=z_1+z_2$.
Given $\Psi_1, \Psi_2 \in \mathcal C_m$ we define  $\Psi \in \mathcal C_m$ by
\begin{equation}\label{eq:DefnControlSum}
\begin{aligned}
\Psi&(y_k,\dots,y_0;\,i_{k-1},\dots,i_0):=\\
&\Psi_1(\phi(k,\sigma,z_1,\Psi_1),\dots,z_1;\,i_{k-1},\dots,i_0))\\
&+
\Psi_2(\phi(k,\sigma,z_2,\Psi_2),\dots,z_2;\,i_{k-1},\dots,i_0)),
\end{aligned}
\end{equation}
for all $(i_{k-1},\dots, i_0)\in \M^k$, all $(y_k,\dots, y_0)\in \R^{n\times (k+1)}$, and all all $k\in \N$.
One can check by induction that, for all $\sigma\in \M^\omega$ and all $k\in \N$, it holds that 
$\phi(k,\sigma,x_0,\Psi)=\phi(k,\sigma,z_1,\Psi_1)+\phi(k,\sigma,z_2,\Psi_2)$.
It follows from the definition of $V$ in~\eqref{eq:lyap} that
$V(x_0) \le V(z_1)+V(z_2)$. Indeed, let us consider any $\varepsilon>0$ and $\Psi_1,\Psi_2\in \cC_m$ such that 
\[
V(z_j)+\frac{\varepsilon}{2}\geq \sup_{\sigma \in  \Sigma^\omega} \sum_{k=0}^{\infty} 
\|\phi(k,\sigma,z_j,\Psi_j)\|, \;\;\forall \;j\in \{1,2\},
\]
and consider $\Psi$ defined as in~\eqref{eq:DefnControlSum}, we have, by triangular inequality, that 
\[
\begin{aligned}
V(x_0)&\leq \sup_{\sigma \in  \Sigma^\omega} \sum_{k=0}^{\infty} 
\|\phi(k,\sigma,x_0,\Psi)\|\\&=\sup_{\sigma \in  \Sigma^\omega} \sum_{k=0}^{\infty} 
\|\phi(k,\sigma,z_1,\Psi_1)+\phi(k,\sigma,z_2,\Psi_2)\|\\&\leq\sup_{\sigma \in  \Sigma^\omega} \sum_{k=0}^{\infty} 
\|\phi(k,\sigma,z_1,\Psi_1)\|\\&\,\;\;+\sup_{\sigma \in  \Sigma^\omega} \hspace{-0.1cm}\sum_{k=0}^{\infty} \|\phi(k,\sigma,z_2,\Psi_2)\|   \leq V(z_1)+V(z_2)+\varepsilon.  
\end{aligned}
\]
By arbitrariness of $\varepsilon>0$ we can thus conclude that $V(x_0)\leq V(z_1)+V(z_2)$, as required.

We have thus proved that $V:\R^n\to \R$ is positive definite, absolutely homogeneous of degree $1$ and sub-additive, in other words it is a \emph{norm} of $\R^n$. This also implies that $V$ is convex and therefore continuous. Moreover, using the dynamic programming principle (see e.g.~\citep[Chapter~1]{bertsekas2012dynamic}), we get that
$$
V(x) = \|x\| 
+\inf_{u\in \R^m} \max_{i\in \Sigma}
V(A_i x + B_i u),\; \forall x \in \R^n.
$$
Since $V$ is a norm in $\R^n$, the infimum is actually a minimum.
Indeed, 
\[
\inf_{u\in \R^m} \max_{i\in \Sigma}
V(A_i x + B_i u)=\inf_{u\in K^\perp} \max_{i\in \Sigma}
V(A_i x + B_i u)
\]
where $K^\perp$ is the linear subspace orthogonal to $K=\bigcap_{i\in I}\text{Ker}(B_i)$, i.e.,  $\R^m=K\oplus K^\perp$.
Then, 
\[
\inf_{u\in \R^m} \max_{i\in \Sigma}
V(A_i x + B_i u)=\inf_{u\in K^\perp} W(u)
\]
where $W:\R^m\to \R$ is the function defined by $W(u)=\max_{i\in I}V_i(A_ix+B_iu)$, which is  continuous, convex and bounded from below (by $0$). It is easy to see that, when restricted to the linear subspace $K^\perp$, such function is also coercive, and thus its infimum is attained.
We can thus write
\begin{equation}
    \label{eq:dp}
V(x) = \|x\| 
+\min_{u\in \R^m} \max_{i\in \Sigma}
V(A_i x + B_i u),\; \forall x \in \R^n.
\end{equation}
Hence, let us consider a feedback control $\Phi:\R^n \to \R^m$ such that 
$$
\Phi(x)\in \argmin_{u\in \R^m} \max_{i\in \Sigma}
V(A_i x + B_i u),\; \forall x \in \R^n.
$$
Since $V$ is absolutely homogeneous of degree $1$ and by linearity of the dynamics, we can always choose $\Phi$ to be homogeneous of degree $1$. 

 For all $x\in \R^n$ we have  
$$
 \max_{i\in \Sigma}
V(A_i x + B_i \Phi(x)) = \min_{u\in \R^m} \max_{i\in \Sigma}
V(A_i x + B_i u).
$$
Then, from \eqref{eq:dp} and \eqref{eq:lyap1} we get
$$
V(A_i x + B_i \Phi(x)) \le V(x)-\|x\| \le \left(1 - \frac{1-\gamma}{C}\right) V(x),
$$
for all $i\in \Sigma$. 
Hence $V$ is a continuous exponential Lyapunov function for the closed-loop, which can be described more compactly by a difference inclusion defined by
\[
x^+\in \{A_ix+B_i\Phi(x)\;\;\vert\;\;i\in \Sigma\}.
\]
Thus, by a standard Lyapunov argument, the closed-loop is UES. This, by definition, implies that system~\eqref{eq:SwitchedSystemInput} is IFS, concluding the proof of Item~\emph{(A)}.

The proof for the mode-dependent case follows essentially the same lines, with 
the main differences as follows.
We suppose that the system is DFS$_m$.
We consider the set of \emph{current-mode dependent controllers with state-signal memory}, defined by 
\[
\cC_m^d:=\{\chi^d:\cH\to \R^m\}.
\]
Given $\Psi^d\in \cC_m^d$, as usual, we denote 
 by $\phi(k,\sigma,x_0,\Psi^d)$ the corresponding solution evaluated at time $k\in \N$.
Then, we can, similarly to the previous case,  consider the function 
$V:\R^n \to \R$ defined for all $x_0 \in \R^n$ by
\begin{equation}
\label{eq:lyap2}
V(x_0) = \inf_{\Psi\in \mathcal C_m^d} \sup_{\sigma \in  \Sigma^\omega} \sum_{k=0}^{\infty} 
\|\phi(k,\sigma,x_0,\Psi)\|.    
\end{equation}
Also in this case, the function $V$ can be shown to be a norm of $\R^n$. 
Moreover, using the dynamic programming principle and by coercivity and continuity of $V$, we get that
$$
V(x) = \|x\| 
+ \max_{i\in \Sigma}
\min_{u_i\in \R^m} V(A_i x + B_i u_i),\; \forall x \in \R^n.
$$
Then, 
consider the mode-dependent feedback control
$\Phi_d:\Sigma\times \R^n \to \R^m$ such that 
$$
\Phi_d(i,x)\in \argmin_{u\in \R^m} 
V(A_i x + B_i u),\;\forall\,i\in \Sigma,\; \forall x \in \R^n.
$$
Then, as in the previous case, $V$ can be shown to be a continuous Lyapunov function for the closed-loop system, which is therefore (DFS).\hfill $\square$
\end{pf}
As a byproduct of the proof of Theorem~\ref{thm:main}, we can prove that, once a feedback controller exists, the (exponential) stability of the closed loop is \emph{robust} with respect to external perturbations, in the sense formalized in~\citep{KellTeel}. 
\begin{cor}[Robustness]
System~\eqref{eq:SwitchedSystemInput} is (IFS) if and only if there exists a feedback map $\Phi:\R^n\to \R^m$ such that the differential inclusion
\[
x^+\in F_\Phi(x):= \{A_ix+B_i\Phi(x)\;\;\vert\;\;i\in \Sigma\} 
\]
is robustly\footnote{For the formal definition of robustness for difference inclusions, we refer to~\citep[Definition 2.3]{KellTeel}} uniformly exponentially stable.\\
System~\eqref{eq:SwitchedSystemInput} is (DFS) if and only if there exists a feedback map $\Phi_d:\Sigma\times \R^n\to \R^m$ such that the difference inclusion
\[
x^+\in F_\Phi^d(x):=\{A_ix+B_i\Phi_d(i,x)\;\;\vert\;\;i\in \Sigma\} 
\]
is robustly uniformly exponentially stable.\\
\end{cor}
\begin{pf}
In the proof of Theorem~\ref{thm:main} we have proved that (IFS) (or, equivalently, (IFS$_m$)) implies the existence of a (homogeneous of degree~$1$) feedback $\Phi:\R^n\to \R^m$ and a continuous function $V:\R^n\to \R$ such that there exist $a_1,a_2\geq 0$ and $\wt \gamma\in [0,1)$ such that
\[
\begin{aligned}
a_1\|x\|\leq V(x)&\leq a_2\|x\|\;\;\;\forall \;x\in \R^n\\
V(f)&\leq \wt \gamma V(x)\;\;\;\forall \;x\in \R^n,\;\;\forall \;f\in F_\Phi(x).
\end{aligned}
\]
By applying \citep[Theorem 2.8]{KellTeel}, since $V$ is continuous, we can conclude that the stability of the closed-loop is  robust, 
 in the sense of ~\citep[Definition 2.3]{KellTeel}.
 The fact that the stability is in particular exponential follows by the fact that the closed loop is homogeneous of degree~$1$, or, equivalently, that the function $V$ is actually a norm of $\R^n$, with linear decrease along the solutions of the considered difference inclusion.

 The mode-dependent case is similar and thus not explicitly developed here.\hfill $\square$
\end{pf}

\begin{rem}
In Theorem~\ref{thm:main} we have shown that, as long as one allows for non-linear feedback maps, the additional information provided by the past behavior of the system is not required. This equivalence is theoretically significant, as it characterizes the feedback stabilization of linear switched systems solely through feedback laws that depend on the current state (and possibly on the currently active mode).

On the other hand, in practical scenarios the use of dynamic controllers, or controllers that exploit the information carried by the past values of the switching signal, can be advantageous. Indeed, it has already been established in the literature~\citep{BlaMia03,BlaMiaSav07,LeeKha09,Lee06} that in such cases the (past-dependent, broadly speaking) feedback maps can be chosen to depend \emph{linearly} on the current state. From a numerical standpoint, this is beneficial because the resulting control-design problem can then be reformulated as a matrix optimization problem. Nevertheless, even the design of non-linear (though piecewise-linear) controllers of the type described in (the proof of) Theorem~\ref{thm:main} can be addressed by means of linear matrix inequalities techniques, for instance as developed in~\cite{DelRosAlv24} (see also references therein).
\end{rem}

\section{Conclusion}\label{sec:conc}

In this paper, we have studied a fundamental question concerning the stabilizability of switched systems under arbitrary switching through a continuous control input: namely, whether memory is actually required to achieve stabilization. Our main result, established through a rigorous development, shows that if a stabilizing controller with infinite memory exists, then a static controller also exists.

This finding clarifies the role of memory-based strategies. While controllers with memory can be useful for obtaining linear dynamic controllers, linear static controllers are known to be conservative in the context of switched systems. The results of this work indicate that memory need not be invoked to achieve stabilizability, although static controllers may need to be piecewise linear and homogeneous of degree one, which is consistent with prior observations in the literature.

Future work may investigate numerical methods for approximately constructing the memoryless controllers identified in this paper, explore extensions to more general classes of switched and hybrid systems for which this equivalence holds, and examine the case of output feedback.



\bibliography{ifacconf}             

@article{lima2025,
  author  = {T. {Alves Lima} and M. {Della Rossa} and A. Girard},
  title   = {Feedback stabilization of switched systems under arbitrary switching: A convex characterization},
  journal = {arXiv preprint arXiv:2506.03759},
  year    = {2025},
  note    = {Available at \url{https://arxiv.org/abs/2506.03759}.}
}

@article{WICKS1998140,
title = {Switched Controller Synthesis for the Quadratic Stabilisation of a Pair of Unstable Linear Systems},
journal = {European Journal of Control},
volume = {4},
number = {2},
pages = {140-147},
year = {1998},
issn = {0947-3580},
author = {M. Wicks and P. Peleties and R. DeCarlo},
}

@book{Lib03,
  title={Switching in Systems and Control},
  author={D.~Liberzon},
  series={Systems \& Control: Foundations \& Applications},
  year={2003},
  publisher={Birkh{\"a}user}
}

@article{BlaMiaSav07,
title = {Stability results for linear parameter varying and switching systems},
journal = {Automatica},
volume = {43},
number = {10},
pages = {1817-1823},
year = {2007},
issn = {0005-1098},
author = {F.~Blanchini and S.~Miani and C.~Savorgnan},
}

@article{FIACCHINI2017181,
title = {Control co-design for discrete-time switched linear systems},
journal = {Automatica},
volume = {82},
pages = {181-186},
year = {2017},
issn = {0005-1098},
author = {M.~Fiacchini and S.~Tarbouriech},
}

@article{BlaMia03,
author = {F.~Blanchini and S.~Miani},
title = {Stabilization of {LPV} Systems: State Feedback, State Estimation, and Duality},
journal = {SIAM Journal on Control and Optimization},
volume = {42},
number = {1},
pages = {76-97},
year = {2003},
}

@article{AhmJun:14,
	author={A.~A.~Ahmadi and R.~M.~Jungers and P.~A.~Parrilo and M.~Roozbehani},
	title={Joint spectral radius and path-complete graph {Lyapunov} functions},
	journal={SIAM Journal on Control and Optimization},
	year={2014},
	volume={52},
	pages={687--717},
	publisher={SIAM}
}

@article{DelRosAlv24,
title = {Graph-based conditions for feedback stabilization of switched and {LPV} systems},
journal = {Automatica},
volume = {160},
pages = {111427},
year = {2024},
issn = {0005-1098},
author = {{Della Rossa}, M. and  {Alves Lima}, T. and Jungers, M. and Jungers, R.~M.}
}

@article{Lee06,
  author={Lee, J.-W.},
  journal={IEEE Transactions on Automatic Control}, 
  title={On Uniform Stabilization of Discrete-Time Linear Parameter-Varying Control Systems}, 
  year={2006},
  volume={51},
  number={10},
  pages={1714-1721},
}

@article{HuMaLIn08,
  title={Stabilization of switched systems via composite quadratic functions},
  author={T.~Hu and L.~Ma and Z.~Lin},
  journal={IEEE Transactions on Automatic Control},
  volume={53},
  number={11},
  pages={2571--2585},
  year={2008},
  publisher={IEEE}
}

@article{GerCol06,
author = { J.C.~Geromel  and  P.~Colaneri},
title = {Stability and stabilization of discrete time switched systems},
journal = {International Journal of Control},
volume = {79},
number = {7},
pages = {719-728},
year  = {2006},
publisher = {Taylor \& Francis},
}

@article{JunMas17,
author = {R.M.~Jungers and P.~Mason},
title = {On Feedback Stabilization of Linear Switched Systems via Switching Signal Control},
journal = {SIAM Journal on Control and Optimization},
volume = {55},
number = {2},
pages = {1179-1198},
year = {2017},
}

@ARTICLE{LeeKha09,
  author={J.-W.~Lee and P.P.~Khargonekar},
  journal={IEEE Transactions on Automatic Control}, 
  title={Detectability and Stabilizability of Discrete-Time Switched Linear Systems}, 
  year={2009},
  volume={54},
  number={3},
  pages={424-437}
}

@article{KellTeel,
author = {C.M.~Kellett and A.R.~Teel},
title = {On the Robustness of {$\mathcal{KL}$}-stability for Difference Inclusions: Smooth Discrete-Time {Lyapunov} Functions},
journal = {SIAM Journal on Control and Optimization},
volume = {44},
number = {3},
pages = {777-800},
year = {2005},
}

@article{DelRosJung24,
author = {M.~{Della Rossa} and R.M.~Jungers},
title = {Multiple {Lyapunov} Functions and Memory: A Symbolic Dynamics Approach to Systems and Control},
journal = {SIAM Journal on Control and Optimization},
volume = {62},
number = {5},
pages = {2695-2722},
year = {2024},
}

@ARTICLE{FiaGir16,
  author={Fiacchini, M. and Girard, A. and Jungers, M.},
  journal={IEEE Transactions on Automatic Control}, 
  title={On the Stabilizability of Discrete-Time Switched Linear Systems: Novel Conditions and Comparisons}, 
  year={2016},
  volume={61},
  number={5},
  pages={1181-1193}
}

@article{DettJun20,
title = {Lower bounds and dense discontinuity phenomena for the stabilizability radius of linear switched systems},
journal = {Systems \& Control Letters},
volume = {142},
pages = {104737},
year = {2020},
author = {Dettmann, C.~P. and  Jungers, R.~M. and Mason, P.},
}

@article{ZhangAbate09,
title = {Exponential stabilization of discrete-time switched linear systems},
journal = {Automatica},
volume = {45},
number = {11},
pages = {2526-2536},
year = {2009},
author = {Zhang, W. and Abate, A. and  Hu, J. and  Vitus, M.~P.},
}

@article{LinAnt08,
author = {Lin, H. and  Antsaklis, P.~J.},
title = {Hybrid state feedback stabilization with $\ell_2$ performance for discrete-time switched linear systems},
journal = {International Journal of Control},
volume = {81},
number = {7},
pages = {1114--1124},
year = {2008},
publisher = {Taylor \& Francis},
}

@book{bertsekas2012dynamic,
  title={Dynamic programming and optimal control: Volume I},
  author={Bertsekas, D.},
  volume={4},
  year={2012},
  publisher={Athena scientific}
}
                                            


\end{document}